\pdfoutput=1 

\RequirePackage[l2tabu, orthodox]{nag}

\documentclass[12pt]{amsart}
\usepackage{colonequals}
\usepackage{appendix}
\usepackage{tikz-cd}
\usepackage[utf8]{inputenc} 
\usepackage{fullpage,url,amssymb,enumerate,colonequals,graphicx}
\usepackage[all]{xy} 
\usepackage{mathrsfs} 

    \DeclareFontFamily{U}{wncy}{}
    \DeclareFontShape{U}{wncy}{m}{n}{<->wncyr10}{}
    \DeclareSymbolFont{mcy}{U}{wncy}{m}{n}
    \DeclareMathSymbol{\Sha}{\mathord}{mcy}{"58} 

\usepackage[T1]{fontenc}
\usepackage[utf8]{inputenc}

\usepackage[dvipsnames,xcdraw,hyperref]{xcolor}

\DeclareMathOperator{\Hilb}{Hilb}

\newcommand{\Aff}{\mathbb{A}}
\newcommand{\C}{\mathbb{C}}

\newcommand{\F}{\mathbb{F}}
\newcommand{\G}{\mathbb{G}}

\newcommand{\PP}{\mathbb{P}}
\newcommand{\Q}{\mathbb{Q}}

\newcommand{\Z}{\mathbb{Z}}

\newcommand{\TT}{\mathbf{T}}
\newcommand{\ti}{\mathbf{t}}

\newcommand{\calL}{\mathcal{L}}

\newcommand{\calN}{\mathcal{N}}
\newcommand{\calO}{\mathcal{O}}

\newcommand{\calT}{\mathcal{T}}

\newcommand{\calV}{\mathcal{V}}

\newcommand{\calX}{\mathcal{X}}

\DeclareMathOperator{\Bl}{Bl}
\DeclareMathOperator{\sep}{sep}
\DeclareMathOperator{\Spec}{Spec}

\DeclareMathOperator{\Aut}{Aut}

\DeclareMathOperator{\Char}{char}

\DeclareMathOperator{\Gal}{Gal}

\DeclareMathOperator{\im}{im}

\DeclareMathOperator{\Pic}{Pic}

\DeclareMathOperator{\Proj}{Proj}

\DeclareMathOperator{\Stab}{Stab}

\DeclareMathOperator{\Sym}{Sym}

\newcommand{\PGL}{\operatorname{PGL}}

\usepackage{amsthm}
\newtheorem{theorem}{Theorem}[section]
\newtheorem{lemma}[theorem]{Lemma}

\newtheorem{corollary}[theorem]{Corollary}
\newtheorem{proposition}[theorem]{Proposition}

\theoremstyle{definition}

\theoremstyle{remark}
\newtheorem{remark}[theorem]{Remark}

\makeatletter
\g@addto@macro\bfseries{\boldmath} 
\makeatother

\usepackage{tikz}
\usepackage{microtype}  

\usepackage[
	pagebackref,
	pdfauthor={Xinyu Fang}, 
]{hyperref}
\usepackage{cleveref}
\begin{document}

\title{Galois group of exceptional curves on the generic del Pezzo surface}
\author{Xinyu Fang}
\address{Department of Mathematics, Harvard University, Cambridge, MA 02138, USA}
\email{xinyufang@math.harvard.edu}
\urladdr{\url{https://sites.google.com/view/xinyufang/}}
\date{Mar 29, 2026}

\begin{abstract}
We prove that the Galois action on the exceptional curves on the generic del Pezzo surface of degree $d$ is maximal for all degrees $d$ and  over any field $k$.
As a consequence of the case $d=3$, we deduce that over $\F_q(u)$, 100\% of cubic surfaces have no Brauer-Manin obstruction.
\end{abstract}

\maketitle
\tableofcontents
\section{Introduction}
\label{sec:intro}
Let $k$ be any field and let $k_s$ be a separable closure.
A \emph{del Pezzo surface} over  $k$ is a smooth projective surface over $k$ with ample anticanonical divisor, which we denote by $-K_X$. Its \emph{degree} is defined to be the self-intersection number $d=(-K_X)^2$. We have the following classical result on the classification of del Pezzo surfaces (see, for example, \cite[Theorem 2.1.1]{VarillyAlvaradoThesis} or \cite[Theorem 24.4]{Manin1986CubicForms}).
\begin{theorem}
    Let $X$ be a del Pezzo surface of degree $d$ over $k$. Then $X_{k_s}$ is isomorphic to the blow-up of $\PP^2_{k_s}$ at $(9-d)$ points in general position in $\PP^2(k_s)$, or $d=8$ and $X$ is isomorphic to $\PP^1_{k_s}\times \PP^1_{k_s}$.
\end{theorem}
Given a smooth projective surface $X$ over $k$, 
an \emph{exceptional curve} on $X$ is an irreducible curve $C\subset X_{\overline{k}}$ such that $C^2=C\cdot K_X=-1$, where $\overline{k}$ is an algebraic closure of $k$. In fact, any such curve is necessarily defined over $k_s$ (see, for example, \cite[Theorem 9.3.1(b)(ii)]{qpoints}). By adjunction, such a curve has arithmetic genus 0, so it is isomorphic to $\PP^1_{k_s}$. The Galois action on $X_{k_s}$ induces an action of $\Gal(k_s/k)$  on the set of exceptional curves on $X_{k_s}$. 

In this paper, we study this Galois action on exceptional curves for the \emph{generic} del Pezzo surface of each degree over $k$ (defined in \Cref{sec:generic}). 
We focus on the cases  $1\leq d\leq 7$, since the cases $d=8,9$ are trivial. 

Given a del Pezzo surface $X$ of degree $1\leq d\leq 7$ over $k$, let $n_d$ be the number of exceptional curves on $X_{k_s}$ and $\Gamma_d$ be the incidence graph for these curves (i.e., the graph whose vertices are the exceptional curves and whose edges record the incidences among them). Write $\Aut(\Gamma_d)$ for the group of incidence-preserving automorphisms of the configuration of exceptional curves on $X_{k_s}$.
We have the following table (see \cite{Manin1986CubicForms}, Theorems 25.4 and 26.2).
\begin{table}[h]
\centering
    \begin{tabular}{c|c|c|c|c|c|c|c}
    \label{tab:aut-groups-of-lines}
       $d$ & 1 & 2 & 3 & 4 & 5 & 6 & 7\\\hline
       $n_d$ & 240 & 56 & 27 & 16 & 10 & 6 & 3 \\\hline
       $\Aut(\Gamma_d)$ & $W(E_8)$ & $W(E_7)$ & $W(E_6)$ & $W(D_5)$ & $W(A_4)\simeq S_5$ & $W(A_2\times A_1)\simeq D_6$ & $S_2$\\\vspace{.05in}
    \end{tabular}
    \caption{Exceptional curves on del Pezzo surfaces and their permutation groups}
\end{table}

Since the Galois action on the exceptional curves on $X$ must preserve their incidences, the image of this action is a subgroup of $\Aut(\Gamma_d)$.
We say that the action is \emph{maximal} if the image is equal to the whole group $\Aut(\Gamma_d)$.
Let $G_d$ denote the Galois image in $\Aut(\Gamma_d)$ for the generic del Pezzo surface of degree $d$ over $k$. Our main result is the following:
\begin{theorem}
\label{thm:main}
    For any field $k$ and any $1\leq d\leq 7$,
    $G_d=\Aut(\Gamma_d)$. In other words, the generic del Pezzo surface has maximal Galois action on its exceptional curves.
\end{theorem}
\begin{remark}
    The main new content in this theorem is the cases $1\leq d\leq 4$ over an arbitrary base field $k$. When $d=7$, the situation is very simple. The cases $5\leq d\leq 6$ follow immediately from \cite[Theorems 1.2 and 1.7]{Zaitsev2023_FormsDelPezzo5_6}, but our proof provides an alternative approach. The case $d=3$ and $k=\C$ is very classical, proven as early as \cite[III.II.V]{Jordan1870}; see also \cite[\S 4]{Todd_1935} and \cite[\S III.3]{Harris1979Galois}. 
    This has been extended to $d=3$ and $\Char k\neq 2,3$ in \cite[Proposition 4.8]{Achter_2013} using an arithmetic Torelli map which identifies cubic surfaces with marked lines with certain principally polarized abelian fivefolds with level structure, and the existence of smooth compactifications of Shimura varieties. See \cite[\S 1.2]{kaya2025inverse} for a survey of some other partial results (for example, over specific fields). In contrast to most of these partial results, our proof is uniform over all fields $k$ and in particular does not depend on the characteristic.
\end{remark}

We deduce a consequence for the Brauer-Manin obstruction on cubic surfaces over $\F_q(u)$.
\begin{corollary}
    100\% of cubic surfaces over $\F_q(u)$, when ordered by height of the coefficients of their equations, have no Brauer-Manin obstruction.
\end{corollary}

From now on, we fix an arbitrary field $k$ unless otherwise specified.

\section{The generic del Pezzo surface}
\label{sec:generic}

In this section, for each $1\leq d\leq 7$, we define a parameter space $\mathscr{H}_d$ for del Pezzo surfaces $X$ over $k$ of degree $d$ with their natural embeddings into the relevant (weighted) projective space. Once we see that these parameter spaces are integral schemes, we define the \emph{generic del Pezzo surface of degree $d$} to be the generic fiber of the universal family over $\mathscr{H}_d$. 

To define $\mathscr{H}_d$, it is convenient to separate into two cases based on the very ampleness of the anticanonical bundle. 

\subsection{Degrees $3\leq d\leq 7$}
When $3\leq d\leq 7$, the anticanonical bundle on $X$ is very ample and gives an embedding
\[
|-K_X|: X\hookrightarrow\PP^d
\]
as a degree $d$ surface. 
We can compute (as in \cite[Lemma III.3.2.2]{kollar1996rational}) using Riemann-Roch for surfaces  and the identities $\chi(\calO_X)=1$  and $(K_X)^2=d$ that any del Pezzo surface of degree $d$ with the anticanonical embedding has
Hilbert polynomial
\[
P_d(m)=\chi(X,\calO_X(-mK_X))=\frac{d}{2}m(m+1)+1. 
\]
Thus, we can define $\mathscr{H}_d$ as a subscheme of the Hilbert scheme $\Hilb_{\PP^n}^{P_d}$, which is a projective scheme parametrizing closed subschemes of $\PP^d$ with Hilbert polynomial $P_d$ (see \cite[Theorem 1.1]{hartshorne-def-thry} or \cite[Theorem 1.4]{kollar1996rational}). The following lemma shows that $\mathscr{H}_d$ is a smooth connected quasiprojective scheme.

\begin{lemma}
\label{lem:param-for-degree-3-7}
    Let $3\leq d\leq 7$.
    The subspace $\mathscr{H}_d\subset \Hilb^{P_d}_{\PP^d}$ parametrizing del Pezzo surfaces of degree $d$ embedded anticanonically into $\PP^d$ is a smooth, connected, open subscheme of $\Hilb^{P_d}_{\PP^d}$.
    \begin{proof}
        We first show that the
         locus $\mathscr{H}_d$ in $\Hilb_{\PP^d}^{P_d}$ is open. It is characterized by the condition that it parametrizes smooth surfaces $X$ with $-K_X\simeq\calO_X(1)$.
        Smoothness is open (see, e.g., \cite[Tag 01V4]{stacks-project-01V4}). We show that the condition $-K_X\simeq \calO_X(1)$ is also open. Let $U\subset \Hilb_{\PP^d}^{P_d}$ denote the open subscheme where the universal family $V\to U$ is smooth.  Consider the relative Picard scheme $\underline{\Pic}_{V/U}$ (for its existence, see \cite{FGAexplained}, Theorem 9.4.8). Since $-K_X\simeq\calO_X(1)$ holds for fibers $X$ that are del Pezzo surfaces, it suffices to show that the Picard scheme is \'etale above the points parametrizing del Pezzo surfaces. For a del Pezzo surface $X$, we have $h^1(X,\calO_X)=h^2(X,\calO_X)=0$ (see \cite[Lemma III.3.2.1]{kollar1996rational}). Given a parameter $t\in U$ for a del Pezzo surface, the 
        above vanishing of cohomology for $X=V_t$ imply that $\underline{\Pic}_{V/U}$ is smooth of relative dimension 0 in an open neighborhood of $t$ (see \cite[Theorem 9.5.11 and Proposition 9.5.19]{FGAexplained}), hence \'etale. Therefore the equality $-K_X\simeq\calO_X(1)$ holds for fibers of $V\to U$ over an open neighborhood of $t$. Thus $\mathscr{H}_d$ is an open subscheme of $U\subset \Hilb_{\PP^d}^{P_d}$, hence  also open.

        Next, we show that $\Hilb^{P_d}_{\PP^d}$ is smooth at each point $[X]\in \mathscr{H}_d$ using the smoothness criterion by Grothendieck \cite{FGA} (see also \cite[Theorem 1.1]{hartshorne-def-thry}), which says that this is true if $X$ is a locally complete intersection and $H^1(X,\calN_{X/\PP^d})=0$. The first condition follows from $X$ being  a smooth surface in $\PP^d$. It suffices to show that $H^1(X,\calN_{X/\PP^d})=0$. To do this, we use the long exact sequence of cohomology associated to the normal bundle sequence
        \[
        0\to\calT_X\to\calT_{\PP^d}|_X\to \calN_{X/\PP^d}\to 0.
        \] 
        Namely,
        \begin{equation}
        \label{eq:LES-for-H^1(N)}
             H^1(X,\calT_{\PP^d}|_X)\to H^1(X,\calN_{X/\PP^d})\to H^2(X,\calT_X).
        \end{equation}
        
        Now by Serre duality, $H^2(X,\calT_X)\simeq H^0(X,\Omega^1_X\otimes\omega_X)^\vee$. 
        In characteristic 0, we can apply the Kodaira-Akizuki-Nakano vanishing theorem with $\calL=\omega_X^{-1}$ ample to get $H^0(X,\Omega^1_X\otimes\omega_X)=0$. Now suppose $\Char k=p>0$ and $k$ is algebraically closed.
        By \cite[Corollaire 2.8]{Deligne1987}, since $\omega_X^{-1}$ is ample and $\dim X=2$, we have that $H^0(X,\Omega^1_X\otimes\omega_X)=0$ as long as $X$ is liftable to $W_2(k)$. In our case, liftability can be shown directly by viewing $X$ as $\PP^2$ blown up at $9-d$ points. Indeed, we can lift the collection of blown up points in $\PP^2$ to a closed codimension 2 subscheme $Z\subset\PP^2_{W_2(k)}$, then form the blowup $\calX=\Bl_Z\PP^2_{W_2(k)}$. This gives the desired lift. Thus, we have shown that regardless of the characteristic, we have $H^2(X,\calT_X)=0$.
        
        For $H^1(X,\calT_{\PP^d}|_X)$, use the dual Euler sequence restricted to $X$:
        \[
        0\to \calO_X\to\calO_X(1)^{\oplus d+1}\to \calT_{\PP^d}|_X\to 0
        \]
        is exact, so we have an exact sequence of cohomology
        \begin{equation}
        H^1(X,\calO_X(1)^{\oplus d+1})\to H^1(X,\calT_{\PP^d}|_X)\to H^2(X,\calO_X).
        \end{equation}
        Now $H^2(X,\calO_X)\simeq H^0(X,K_X)^\lor=0$, and $H^1(X,\calO_X(1))=0$  by Kodaira vanishing (true in all characteristics for any nice surface not of general type or quasi-elliptic by \cite[Theorem II.1.6]{Ekedahl}). These imply that $H^1(X,\calT_{\PP^d}|_X)=0$.
        
        Thus, going back to \eqref{eq:LES-for-H^1(N)}, we have 
        \[
        H^1(X,\calN_{X/\PP^d})=0,
        \]
        so $\mathscr{H}_d$ is smooth at $[X]$.

        We now show  $\mathscr{H}_d$ is also irreducible. Consider the open subscheme $W\subset(\PP^2)^{9-d}$ parametrizing $(9-d)$ points in $\PP^2$ in general position (i.e., no three on a line and no six on a conic). Then there is a family of del Pezzo surfaces $\pi:D\to W$ corresponding to the blow-ups of $\PP^2$ at $(9-d)$ points. Let  $\omega_{D/W}^{-1}$ be the relative anticanonical sheaf on $D$, and let $\mathcal{V}=(\pi_*\omega_{D/W}^{-1})^{\oplus (d+1)}$, which is a locally free sheaf on $W$ corresponding to the vector bundle $\mathbb{V}(\calV)=\Spec_W(\Sym^\bullet \calV^\vee)$. At each closed point $w\in W$, the fiber of this vector bundle is
        \[
        H^0(D_w, -K_{D_w})^{\oplus (d+1)}.
        \]
        There is a nonempty open subscheme inside $\mathbb{V}(\calV)$ that parametrizes pairs $(w,(b_i))$ of a del Pezzo surface $D_w$ and an ordered basis $(b_i)\in H^0(D_w, -K_{D_w})^{\oplus (d+1)}$ which gives an anticanonical embedding of $D_w$ into $\PP^d$. Since every del Pezzo surface of degree $3\leq d\leq 7$ is geometrically the blow-up of $\PP^2$ at $(9-d)$ points, this open subscheme of $\mathbb{V}(\calV)$ maps to $\mathscr{H}_d$, and the map is  surjective on $\overline{k}$-points. Now since $W$ is irreducible, the vector bundle $\mathbb{V}(\calV)$  over $W$ is also irreducible. This implies that $\mathscr{H}_d$ is irreducible.
    \end{proof}
\end{lemma}
\begin{remark}
\label{rem:deg-3-4-parameter-space}
    For $d=3,4$, the del Pezzo surface is a complete intersection of hypersurfaces in $\PP^d$, so a parameter space can also be given directly as a certain open subvariety of the parameter space of such hypersurfaces. More precisely, for $d=3$, we have ${\mathscr{H}_d}\subset \PP^{19}$ is the open subscheme parametrizing smooth cubic forms on $\PP^3$; for $d=4$, we have that ${\mathscr{H}_d}$ is the open subscheme of the Grassmannian of two-dimensional subspaces of the space of quadratic forms on $\PP^4$ parametrizing smooth complete intersections of two quadrics in $\PP^4$.
    The construction above agrees with these alternative descriptions. 
\end{remark}

\subsection{Degrees 1 and 2}
\label{sec:deg-1-2-param-space}
When $d=1$ or $2$, $-K_X$ is no longer very ample. However, we have
\[
X\simeq \Proj R(X,-K_X)
\]
where
\[
R(X,-K_X)=\sum_{m\geq 0}H^0(X,\calO(-mK_X))
\]
is the anticanonical ring of $X$. This description realizes 
$X$ as a smooth hypersurface in a weighted projective space via the embedding given by some generators of $R(X,-K_X)$.

\begin{theorem}[\cite{kollar1996rational}, Theorem III.3.5]
\label{thm:dP12-as-weighted-hypersurface}
    Let $d\in\{1,2\}$. Then a del Pezzo surface $X$ of degree $d$ over $k$ can be described as follows.
    \begin{itemize}
        \item If $d=2$, then $X\subset\PP(1,1,1,2)$ is a hypersurface cut out by a weighted degree 4  polynomial;
        \item If $d=1$, then $X\subset\PP(1,1,2,3)$ is a hypersurface cut out by a weighted degree 6  polynomial.
    \end{itemize}
\end{theorem}
It is easy to check that $X$ must be contained inside the smooth locus of the relevant weighted projective space.
\begin{lemma}[\cite{kollar1996rational}, V.1.3.7]
Assume the weighted projective space $\PP(a_0,\dots,a_n)$ has isolated singularities. Let $m\geq 1$ be divisible by $a_i$ for all $i$. Then the smooth members of $|\calO(m)|$ form a dense open set. If $H\in |\calO(m)|$ is smooth, then $K_H=\calO(H-\sum a_i)|_H$.
\end{lemma}
Note that both $\PP(1,1,1,2)$ and $\PP(1,1,2,3)$ have isolated singularities. Combining the two statements above, we can define $\mathscr{H}_d$ for $d\in \{1,2\}$  to be the dense open subscheme of the complete linear system $|\calO_{\PP}(m)|$ parametrizing smooth weighted hypersurfaces of degree $m$, where the relevant weighted projective space $\PP$ and the degree $m$ are given in Theorem \ref{thm:dP12-as-weighted-hypersurface}. These are again integral quasiprojective schemes.

\section{Monodromy formulation and irreducibility of the incidence scheme}
\label{sec:monodromy}
Let $K$ be the function field of $\mathscr{H}_d$. Recall that our goal is to compute the image $G_d$ of the action of $\Gal({K}^{\sep}/K)$ on the exceptional curves on the generic del Pezzo surface of degree $d$ (as defined in \S \ref{sec:generic}). In this section, we reformulate this question in terms of the monodromy of a certain finite \'etale cover $\mathscr{Y}_d$ of $\mathscr{H}_d$ whose fibers parametrize the exceptional curves. We show that this covering space is irreducible for $1\leq d\leq 6$, which is equivalent to saying that the monodromy action is transitive.  

\subsection{Definition of $\mathscr{Y}_d$}
\label{subsec:def-of-Yd}
For each $1\leq d\leq 7$, there is a degree $n_d$ \'etale cover 
\[\pi_d:\mathscr{Y}_d\to \mathscr{H}_d\] whose fibers correspond to the $n_d$ exceptional curves on the del Pezzo surface. Concretely, $\mathscr{Y}_d$ is the incidence scheme
\[
\mathscr{Y}_d\colonequals \{(L, S):L\subset S\}\subset \G_d\times \mathscr{H}_d
\]
where for $3\leq d\leq 7$, $\G_d=\G(1,d)$ is the Grassmannian of lines in $\PP^d$; for $1\leq d\leq 2$, $\G_d$ also parametrizes certain lines in the  weighted projective space. In the latter case, it is defined more precisely as follows.

Let $\PP\colonequals\PP(a_0,\dots,a_n)$ be a weighted projective space with weights $(a_0,\dots,a_n)$ and $U\subset\PP$ be its smooth locus. 
By a \emph{line} in $\PP$ we mean the image of a morphism
\[
f:\PP^1\to \PP
\]
such that $\im(f)\subset U$ and $\deg f^{*}\calO_{\PP}(1)=1$. \footnote{Note that even though $\calO_{\PP}(1)$ is generally not an invertible sheaf on $\PP$, its restriction to the smooth locus of $\PP$ is invertible, so it pulls back along $f$ to an invertible sheaf on $\PP^1$. Whenever we write $f^*\calO_{\PP}(1)$, we implicitly mean $f^*(\calO_{\PP}(1)|_U)$.} 
The exceptional curves on a del Pezzo surface of degree $1$ or 2 with the embedding into weighted projective space are all lines in this sense. 

\begin{lemma}
\label{lem:exceptional-curves-are-lines-in-P}
    Suppose $X\subset\PP$ is a del Pezzo surface of degree $1\leq d\leq 7$ embedded into some (weighted) projective space as in \S \ref{sec:generic}, and $f:\PP^1\hookrightarrow X\hookrightarrow \PP$ is the inclusion of some exceptional curve on $X$. Then $\deg f^*\calO_{\PP}(1)=1$.
    \begin{proof}
        The case when $d\geq 3$ is trivial, but the following argument works for any $1\leq d\leq 7$.
        By \cite[V.1.3.7]{kollar1996rational}, we have that 
        \[
        K_X=\calO_{\PP}(d-\sum a_i)|_X=\calO_{\PP}(-1)|_X
        \]
        where $a_i$ are the weights on $\PP$. Let $L=\im(f)$. By first pulling back $\calO_{\PP}(1)$ to $X$, we have
        \[\deg f^*\calO_{\PP}(1)=L\cdot (-K_X)=-L\cdot K_X=-1,\] 
        where the last inequality follows from  $L^2=-1$ and $2g_L-2=L\cdot (L+K_X)$ by adjunction.
    \end{proof}
\end{lemma}

We can parametrize such morphisms $f:\PP^1\to\PP$ using the following lemma.

\begin{lemma}[lines in weighted projective space, cf. \cite{kollar1996rational}, V.1.3.10]
\label{lem:lines-in-wps}
    Suppose $f:\PP^1\to \PP$ has image inside the smooth locus of $\PP$. Then $f$ is given by 
    \[
    f=(f_0,\dots,f_n):\PP^1\to \PP, \quad f_i\in H^0(\PP^1,\calO_{\PP^1}(da_i))
    \]
    where $d=\deg f^*\calO_{\PP}(1).$
    \begin{proof}
        Write $\PP=\Proj k[T_0,\dots,T_n]$ with $\deg T_i=a_i$.
         Suppose $f^*\calO_{\PP}(1)=\calO_{\PP^1}(d)$, then for any $a\in\Z$, $f^*\calO_{\PP}(a)=f^*(\calO_{\PP}(1))^{\otimes a}=\calO_{\PP^1}(da)$, so the map is given by $f=(f_0,\dots,f_n)$ where $f_i=f^\#T_i\in H^0(\PP^1,\calO_{\PP^1}(da_i))$ since $T_i\in H^0(\PP,\calO_{\PP}(a_i))$. 
    \end{proof}
\end{lemma}
In particular, lines in $\PP$ are images of morphisms of the form $f=(f_0,\dots,f_n)$ with $f_i\in H^0(\PP^1,\calO_{\PP^1}(a_i))$.
We show that, moreover, those appearing as exceptional curves in a del Pezzo surface embedded in $\PP$ must all be equivalent up to automorphisms of $\PP$.
\begin{lemma}
\label{lem:equivalence-of-lines}
    Keep the assumptions in Lemma \ref{lem:exceptional-curves-are-lines-in-P}. Then there is an automorphism of $\PP$ taking $\im(f)$ to the line $L_0$ defined by $g=(u,v,0,\dots,0)$ where $u,v$ are the coordinates on $\PP^1$.
    \begin{proof}
        This is trivial when $3\leq d\leq 7$, since all lines in a straight projective space are projectively equivalent. 
        Let us treat the case $d=1$ or $2$. Write $f=(f_0,f_1,f_2,f_3)$ where $f_i\in H^0(\PP^1,\calO_{\PP^1}(a_i))$. It suffices to show that if $\im(f)$ is an exceptional curve on a del Pezzo surface, then we must have that $f_0,f_1$ are linearly independent (possibly after permuting $f_0,f_1,f_2$ if $d=2$), since in that case the automorphism of $\PP=\Proj(x,y,z,w)$ given by 
        \[
        x\mapsto f_0(x,y), \; y\mapsto f_1(x,y),\; z\mapsto z+ f_2(x,y),\; w\mapsto w+ f_3(x,y)
        \] 
        takes the line $L_0$ to $\im(f)$.
        
        In the case $d=2$, the condition that the image of  $f$ does not pass through the singular point $(0,0,0,1)$ of $\PP=\PP(1,1,1,2)$ already forces $f_0,f_1,f_2$ to span a two-dimensional space. So by permuting the first three coordinates, we can assume that $f_0,f_1$ are linearly independent.
        
        In the case $d=1$, if $f_0,f_1$ were not linearly independent, then the image of this line under the rational map
        \[
        |-K_X|=(x:y):X\dashrightarrow \PP^1 
        \]
        sends $L$ to a point. However, $(x:y)$ restricted to its regular locus has fibers whose closures $D$ are in the linear system $|-K_X|$. We know $(-K_X)\cdot D=(-K_X)^2=1$, so $D$ must be irreducible (recall $-K_X$ is ample). 
        This implies that $L$ is such a fiber and $L^2=1$. This shows that $L$ could not be an exceptional curve, a contradiction. 
    \end{proof}
\end{lemma}
Based on the proof above, we can define, for $d=1$ and $2$, the parameter space $\G_d$ for lines in $\PP$ appearing as exceptional curves as follows. For $d=2$, we saw that $(f_0,f_1,f_2)$ must define a line in $\PP^2$, so we can take $\G_2$ to be the vector bundle over $(\PP^2)^\lor$ whose fiber above each line $L\in (\PP^2)^\lor$ is $H^0(L,\calO_L(2))$, which parametrizes possible choices of $f_3$. For $d=1$, precomposing by $\PGL_2$ brings any $(f_0,f_1)$ to $(u,v)$, so we can take $\G_1=H^0(\PP^1,\calO_{\PP^1}(2))\times H^0(\PP^1,\calO_{\PP^1}(3))\simeq\Aff^7$, parametrizing possible choices of $(f_2,f_3)$. We check at once that any exceptional curve occurs as a unique point in these parameter spaces. 

\subsection{Irreducibility of $\mathscr{Y}_d$} 
\begin{proposition}
\label{prop:Yd-irred}
For all $1\leq d\leq 6$,
    $\mathscr{Y}_d$ is geometrically irreducible.
    \begin{proof}
    By base changing, assume $k=\overline{k}$. 

    First consider the case $1\leq d\leq 4$. Recall that $\mathscr{H}_d$ is an open subscheme of some projective scheme $\overline{\mathscr{H}_d}$ as given in Remark \ref{rem:deg-3-4-parameter-space} and at the end of Section \ref{sec:deg-1-2-param-space} (either some projective space or a Grassmannian). The analogous incidence scheme  $I_d=\{(L, S):L\subset S\}\subset \G_d\times \overline{\mathscr{H}_d}$ has first projection $p_1: I_d\to\G_d$ being a projective morphism. Moreover, $p_1$ has isomorphic fibers, because it is equivariant for the action of $\Aut(\PP)$ on both $I_d$ and $\G_d$, and the action is transitive on $\G_d$.
    Thus, to show that $I_d$ is irreducible, it suffices to show that each fiber is irreducible. Fix $L\in \G_d$. For $1\leq d\leq 3$, $p_1^{-1}(L)\subset\overline{\mathscr{H}_d}$ is an irreducible linear subspace cut out by the condition that the hypersurface in $\PP$ contains $L$. In the case $d=4$, we may consider this linear condition first on the open subscheme of $(\PP^{14})^2$ parametrizing pairs of linearly independent quadratic forms on $\PP^4$, see that it cuts out some irreducible subscheme, and then deduce that the fiber in $\overline{\mathscr{H}_d}$ is irreducible since the locus in $(\PP^{14})^2$ surjects onto it. 
    Thus $I_d$ is irreducible, and $\mathscr{Y}_d$ being an open subscheme is also irreducible.

    When $5\leq d\leq 6$, there is a single isomorphism class of del Pezzo surfaces of degree $d$, and since they are all given the anticanonical embedding,  $\PGL_{d+1}$ acts transitively on $\mathscr{H}_d$. Moreover, if $X$ is a del Pezzo surface of degree $5\leq d\leq 6$, then $\Aut(X)$ acts transitively on the $n_d$ exceptional curves on $X$ (for $d=5$, see \cite[Lemma 9.4.21]{qpoints}; for $d=6$, see \cite[Theorem 8.4.2]{DolgachevCAG}) and embeds as a subgroup of $\PGL_{d+1}$. Thus, $\PGL_{d+1}$ acts transitively on $\mathscr{Y}_d$, so for any $(\ell_0,S_0)\in \mathscr{Y}_d$, we have that 
    \[
    \mathscr{Y}_d\simeq \PGL_{d+1}/\Stab(\ell_0,S_0)
    \]
    is the homogeneous space of an irreducible algebraic group, hence irreducible.
    \end{proof}
\end{proposition}

\subsection{Monodromy formulation}
Since the finite \'etale cover $\pi_d:\mathscr{Y}_d\to\mathscr{H}_d$ has geometric fibers isomorphic to the set of exceptional curves on a del Pezzo surface of degree $d$, the monodromy action defines a homomorphism from the \'etale fundamental group of $\mathscr{H}_d$ to the group of automorphisms of the incidence graph of the exceptional curves, i.e., 
\[
    \phi_d:\pi_1^{\text{\'et}}(\mathscr{H}_d,x)\to \Aut(\Gamma_d),
\]
where $x\in \mathscr{H}_d$ is a geometric point. Since  $\mathscr{H}_d$ is irreducible,  the image $\im(\phi_d)\subset\Aut(\Gamma_d)$  does not depend on the choice of the base point up to conjugacy. 
We can identify $\im(\phi_d)$ with $G_d$ via the following general lemma.

\begin{lemma}[cf. \cite{Szamuely2009}, Theorem 5.4.2]
\label{lem:identification-of-galois-groups}
    Suppose $\pi: Y\to H$ is a finite \'etale cover of integral quasiprojective schemes over $k$. Write $\eta:\Spec K\to H$ for the generic point of $H$, where $K$ is the function field of $H$. Let $x:\Spec K^{\sep}\to H$ be a geometric generic point.  Then the action of $\Gal(K^{\sep}/K)\simeq\pi_1^{\text{\'et}}(\eta,x)$ on the geometric fiber $\pi^{-1}(x)$ factors through the natural action of $\pi_1^{\text{\'et}}(H,x)$ on $\pi^{-1}(x)$.
\end{lemma}
\begin{corollary}
\label{lem:identification-of-galois-groups}
    Up to conjugacy, the subgroup $\im(\phi_d)\subset\Aut(\Gamma_d)$ is equal to  $G_d$.
\end{corollary}
Using this alternative description of $G_d$, we obtain the transitivity of $G_d$ as a corollary of Proposition \ref{prop:Yd-irred}.
\begin{corollary}
\label{cor:transitivity-of-Gd}
    The action of $G_d$ is transitive on the set of exceptional curves $\Gamma_d$.
\end{corollary}

\section{Proof of the main theorem}
\label{sec:proof}
We now complete the proof of Theorem \ref{thm:main} via downward induction on the degree $d$. Write $|\Gamma_d|$ for the set of exceptional curves on a del Pezzo surface of degree $d$.
\begin{lemma}
\label{lem:induction}
    Let $1\leq d\leq 6$ and $K$ be an extension of $k$. Suppose there exists a del Pezzo surface $X'$ of degree $d+1$ over $K$  on which the Galois group of the exceptional curves is maximal. Then 
    $G_d$ contains the stabilizer of some element of $|\Gamma_d|$ in $\Aut(\Gamma_d)$, which is isomorphic to $\Aut(\Gamma_{d+1})$.
    \begin{proof}
        Let $L$ be the function field of $X'$. The generic point of $X'$ gives an $L$-point $P$ of the base change $X'_L$.
        We can blow up $X'_L$ at $P$ to obtain a del Pezzo surface $Z$ of degree $d$ over $L$ with a rational line. The image of the action of $\Gal(L^{\sep}/L)$ on the exceptional lines on $Z$ fixes this rational line and is equal to image of the Galois action on the exceptional curves on $X'_L$. Since $L\cap K^{\sep}=K$, 
         this is equal to the image of the action of $\Gal(K^{\sep}/K)$ on the exceptional lines on $X'$, which is $\Aut(\Gamma_{d+1})$ by assumption.   
        This realizes $\Aut(\Gamma_{d+1})$ as a subgroup of the Galois group of the exceptional curves of the generic del Pezzo surface of degree $d$ over $L$, which is a subgroup of $G_d$ by Lemma \ref{lem:passing-to-field-extension} applied to the extension $L/k$. Moreover, this subgroup is the stabilizer in $\Aut(\Gamma_d)$ of a line in $|\Gamma_d|$ by construction.
    \end{proof}
\end{lemma}

\begin{lemma}
\label{lem:passing-to-field-extension}
Let $L/k$ be any field extension.
    Let $G_d'$ be the image of 
    \[
    \phi_{d,L}:\pi_1^{\text{\'et}}(\mathscr{H}_{d,L},x')\to \Aut(\Gamma_d),
    \]
    where $\phi_{d,L}$ is the base change of $\phi_d$ to $L$, $x'$ is a geometric point of $\mathscr{H}_{d,L}$, and $\pi_1^{\text{\'et}}(\mathscr{H}_{d,L},x')$ acts on the geometric fibers of 
    \[
    \pi_{d,L}: \mathscr{Y}_{d,L}\to \mathscr{H}_{d,L}.
    \]
    Then we have $G_d'\subset G_d$.
    \begin{proof}
        This follows from the following commutative diagram with the obvious homomorphisms:
       \[\begin{tikzcd}
            \pi_1^{\text{\'et}}(\mathscr{H}_{d,L},x') \arrow[r]\arrow[d] & \Aut(\Gamma_d)\\
            \pi_1^{\text{\'et}}(\mathscr{H}_d,x')\arrow[ru] &
        \end{tikzcd}.\]
    \end{proof}
\end{lemma}
\begin{theorem}
    For any field $k$ and any $1\leq d\leq 7$,
    we have $G_d=\Aut(\Gamma_d)$. 
\end{theorem}
\begin{proof}
    We first prove the theorem for $d=7$, and then proceed by downward induction on $d$. Recall from Table \ref{tab:aut-groups-of-lines} that $\Aut(\Gamma_7)=S_2$, which permutes the two skew exceptional curves. To show that $G_7=S_2$, we can take a field extension $K/k$ with a degree 2 Galois extension $L/K$, then the blowup $X':=\Bl_P(\PP^2_K)$ at a closed point $P\in \PP^2_K$ with residue field $L$ is a del Pezzo surface of degree $7$ defined over $K$ and has Galois action $S_2$ on the exceptional curves. This implies $G_7=S_2$ by Lemma \ref{lem:passing-to-field-extension}.
    
    Now we explain the induction step. Recall that for any finite group $G$ acting on a set $S$, if $H\subset G$ is a subgroup acting transitively on $S$ and contains the stabilizer of some element $s\in S$, then $H=G$. Apply this with $G=\Aut(\Gamma_d)$, $S=|\Gamma_d|$, and $H=G_d$. By Corollary \ref{cor:transitivity-of-Gd}, $G_d$ acts transitively on $|\Gamma_d|$, so it suffices to show that $G_d$ contains the stabilizer of a line in $|\Gamma_d|$, which is just a copy of $\Aut(\Gamma_{d+1})$. We conclude by Lemma \ref{lem:induction} and the base case $d=7$.
\end{proof}
\section{Consequence for Brauer-Manin obstruction on cubic surfaces}
\label{sec:BM}

For a global field $k$, the action of  $\Gal(k_s/k)$ on the 27 lines on a cubic surface $X$ over $k$ is related to the Brauer-Manin obstruction on $X$ as follows.
\begin{proposition}[\cite{Swinnerton-Dyer_1993}]
\label{prop:full-galois-implies-no-BM}
    Let $k$ be a global field. Let $X$ be a smooth cubic surface over $k$ such that the action of $\Gal(k_s/k)$ on the 27 lines in $X_{k_s}$ is the full $W(E_6)$. Then $H^1(k, \Pic X_{k_s})=0$. In particular, there is no Brauer-Manin obstruction.
    \begin{proof}
        Let $K$ be the smallest (finite) extension of $k$ over which the 27 lines on $X_{k_s}$ are defined. Then the Galois group $G=\Gal(K/k)$ embeds into $W(E_6)$ via its action on the 27 lines by the minimality of $K$. On the other hand, $\Pic X_{k_s}=\Pic X_K$ and $H^1(k,\Pic X_{k_s})=H^1(G,\Pic X_K)$. So $H^1(k,\Pic X_{k_s})$ is killed by $|G|$, which divides $|W(E_8)|=2^7\cdot3^4\cdot5$. By Lemma 1 of \cite{Swinnerton-Dyer_1993}, if the order of $H^1(k,\Pic X_{k_s})$ is divisible by 2, then there is a double-six $S$ on $X_{k_s}$ such that $G$ maps $S$ into itself. By Lemma 4 of \cite{Swinnerton-Dyer_1993}, if the order of $H^1(k,\Pic X_{k_s})$ is divisible by 3, then there is a triple-nine on $X_{k_s}$ such that $G$ maps each nine into itself. Neither of these can happen if $G=W(E_6)$. Moreover, Lemma 9 of \cite{Swinnerton-Dyer_1993} says that $H^1(k,\Pic X_{k_s})$  has order not divisible by 5. This shows that $H^1(k,\Pic X_{k_s})$  must be trivial.
    \end{proof}
\end{proposition}
\begin{remark}
    Swinnerton-Dyer worked over a number field $k$; however, his computation of the Galois action on $\Pic X_{k_s}$ holds over  any global field $k$, and hence also the conclusion on  $H^1(k,\Pic X_{k_s})$.
\end{remark}

When $k=\Q$, Poonen and Voloch (\cite{PoonenVoloch2004-Random}, Proposition 3.4) observed that the maximality of the Galois action on the generic cubic surface implies the maximality of the action for 100\% of cubic surfaces over $\Q$ (via Hilbert's irreducibility theorem), and hence these surfaces have no Brauer-Manin obstruction by Proposition \ref{prop:full-galois-implies-no-BM}.
Now that we have the maximality of the Galois action in arbitrary characteristics, we can extend this result to all rational function fields via the following effective version of Hilbert's irreducibility for rational function fields.

\begin{theorem}[cf. \cite{BarySorokerEntin2021_ExplicitHilbert}, Corollary 3.5]
\label{thm:hilbert-function-field}
    Let $k=\F_q(u)$ be any rational function field.
    Let $f(T_1,\dots,T_n,X)=f(\TT,X)\in k[\TT,X]$ be an irreducible polynomial with coefficients in $\F_q[u]$ that is separable and of positive degree in $X$. Let $G$ be the Galois group of the splitting field of $f$ over $k(\TT)$ as a polynomial in $X$, and for each $\ti\in \F_q[u]^n$, let $G_{\ti}$ be the Galois group for the splitting field of $f(\ti,X)$ over $k$. Let $N(d)$ be the number of $\ti\in \F_q[u]^n_{\leq d}$ such that $G_{\ti}\simeq G$. Then
    \[
    \frac{N(d)}{q^{n(d+1)}}=1+O_f(q^{-\frac{d+1}{2}+r}(d+1)),
    \]
    where $r\leq \deg_X(f)$ is the degree of the constant field extension of $f$. 
    \begin{proof}
        The statement is essentially identical to \cite[Corollary 3.5]{BarySorokerEntin2021_ExplicitHilbert}, except we do the counting in $\F_q[u]^n_{\leq d}$, the set of $n$-tuples of polynomials with degree at most $d$, as opposed to the set of $n$-tuples of monic degree $d$ polynomials.
        The proof follows the same line of argument as \cite[Corollary 3.5]{BarySorokerEntin2021_ExplicitHilbert}, except when applying \cite[Theorem 3.2]{Hsu1996LargeSieve}, in their notation, we consider the intersection of $X$ with $B(0,d)$ instead (see the footnote for the proof of \cite[Proposition 3.4]{BarySorokerEntin2021_ExplicitHilbert}).
    \end{proof}
\end{theorem}

\begin{theorem}
    Let $k=\F_q(u)$. Then
    100\% of cubic surfaces over $k$, when ordered by height  of the coefficients of their equations, have no Brauer-Manin obstruction.
    \begin{proof}
        Consider the finite \'etale cover 
        \[
        \pi_3: \mathscr{Y}_3\to \mathscr{H}_3\subset\PP^{19}
        \]
        defined in Subsection \ref{subsec:def-of-Yd}. By Proposition \ref{prop:full-galois-implies-no-BM}, it suffices to show that the subset $\Omega$ of $k$-points of $\mathscr{H}_3\subset\PP^{19}$ whose fiber has Galois group equal to $\Aut(\Gamma_3)=W(E_6)$ has density equal to 1, when $\PP^{19}(k)$ is ordered by height (in this case, the degrees of the homogeneous coordinates as polynomials in $\F_q[u]$). By Theorem \ref{thm:main}, the generic fiber of $\pi_3$ has Galois group $W(E_6)$, so the conclusion follows from Theorem \ref{thm:hilbert-function-field}. 
    \end{proof}
\end{theorem}

\section{Acknowledgement}
The author thanks Bjorn Poonen for many helpful suggestions.
\bibliographystyle{amsalpha}
\bibliography{ref}
\end{document}